\documentclass[11pt]{amsart}
\usepackage[left=2cm,top=2cm,right=3cm,nofoot]{geometry}
\usepackage{amsmath,mathtools}%
\usepackage{amsfonts}%
\usepackage{amssymb}%
\usepackage{graphicx}
\usepackage{esint}
\usepackage{hyperref}
\usepackage{enumitem}
\usepackage{accents}
\usepackage{etoolbox}
\patchcmd{\subsection}{-.5em}{.5em}{}{}
\newtheorem{theorem}{Theorem}[section]
\theoremstyle{plain}

\newtheorem{corollary}[theorem]{Corollary}

\newtheorem{lemma}[theorem]{Lemma}

\newtheorem{remark}[theorem]{Remark}

\numberwithin{equation}{section}
\theoremstyle{plain}

\usepackage{etoolbox}
\AtEndEnvironment{proof}{\setcounter{claim}{0}}

\begin{document}

\title[]{Spinorial Yamabe-type equations and the B\"ar-Hijazi-Lott invariant }

\author{ Jurgen Julio-Batalla}
\address{ Universidad Industrial de Santander, Carrera 27 calle 9, 680002, Bucaramanga, Santander, Colombia}
\email{ jajuliob@uis.edu.co}
\thanks{  The author was supported by project 3756 of Vicerrector\'ia de Investigaci\'on y Extensi\'on of Universidad Industrial de Santander}

\begin{abstract} We consider on a closed Riemannian spin manifold $(M^n,g,\sigma)$ the spinorial Yamabe type equation $D_g\varphi=\lambda|\varphi|^{\frac{2}{n-1}}\varphi$, where $\varphi$ is a spinor field and $\lambda$ is a positive constant. For a normalized solution $\varphi$ of this equation we find a positive lower bound for $\lambda^2$. As an application we obtain an explicit lower bound of the B\"ar-Hijazi-Lott invariant for some spin manifolds with positive scalar curvature. 
\end{abstract}

\maketitle

\section{Introduction}
Let $(M^n,g,\sigma)$ be a closed  Riemannian spin  manifold with a spin structure $\sigma:P_{spin}(M)\rightarrow P_{SO}(M)$. We consider the representation $\rho:spin(n)\rightarrow Aut(S_n)$ of the spin group compatible with Clifford multiplication and the complex vector bundle $SM=P_{spin}(M)\times_{\rho}S_n$ associated to the  principal bundle $\sigma:P_{spin}(M)\rightarrow P_{SO}(M)$.   This vector bundle is known as the spinor bundle and it is endowed with a natural spinorial Levi-Civita connection $\nabla$, a pointwise Hermitian scalar product $\langle\cdot,\cdot\rangle$ and an elliptic differential operator $D_g$ (called Dirac operator).
 
In this paper, we are interested in the spinorial Yamabe-type equation 
\begin{equation}\label{Dirac}
D_g\varphi=\lambda|\varphi|^{\frac{2}{n-1}}\varphi,
\end{equation} 
where $\varphi\in\Gamma(SM)$, $\lambda\in\mathbb{R}$ and the scalar curvature of the metric $g$, $sc_g$, is positive.

When $\lambda=1$ this equation is referred as the spinorial Yamabe equation (\cite{AmH};\cite{AHM};\cite{AmAubin}). Solutions of this equation are related to some interesting geometric problems. For instance when $n=2$ solutions of the spinorial Yamabe equation produce conformally immersed constant mean curvature surfaces in $\mathbb{R}^3$ (\cite{Friedrich,AmH}), and when $n\geq 2$ it is related to conformally immersed constant mean curvature hypersurfaces in $\mathbb{R}^{n+1}$. Also there is an interesting relation with a conformal spinorial invariant. Indeed, in his Habilitation \cite{AmH} B. Ammann  introduced the B\"ar-Hijazi-Lott invariant $\lambda_{min}^+(M^n,[g],\sigma)$ of the spin manifold $(M,g,\sigma)$ defined by

$$\lambda_{min}^+(M^n,[g],\sigma):=\inf\limits_{g'\in[g]}\lambda_1(D_{g'})vol(M,g')^{1/n}, $$
where $\lambda_1(D_{ g'})$ is the first positive eigenvalue of Dirac operator $D_{g'}$ with respect to the conformal metric $g'$ and $vol(M,g')$ is the volume of $M$ with respect to the metric $g'$.

J. Lott proved in \cite{Lott} that this invariant is positive when the Dirac operator $D_{g'}$ (for some conformal metric $g'$) is invertible. Later B. Ammann in (\cite{Amlow}) extended the result
to any closed spin Riemannian manifold.

In order to obtain generalized metrics which realize the B\"ar-Hijazi-Lott invariant, B. Ammann proved (\cite{AmH}) that this problem  is equivalent to the existence of a spinor field $\varphi$  minimizing the Rayleigh type quotient
\begin{equation}\label{Functional}
J(\psi):=\dfrac{\left(\int_{M}|D_g\psi|^{\frac{2n}{n+1}}dv_g\right)^{(n+1)/n}}{|\int_{M} \langle D_g\psi,\psi\rangle dv_g|}
\end{equation}
on the space of smooth spinor fields. Here $dv_g$ is the volume element of $(M^n,g)$.

Moreover, the infimum of $J$ is $\lambda_{min}^+(M^n,[g],\sigma)$ and the corresponding Euler-Lagrange equation is 
$$D_g\varphi=|\varphi|^{\frac{2}{n-1}}\varphi\quad\text{with}\quad\int_M|\varphi|^{\frac{2n}{n-1}}dv_g= \lambda_{min}^+(M^n,[g],\sigma)^n$$
or equivalently
\begin{equation}\label{ELequation}
D_g\varphi=\lambda_{min}^+(M^n,[g],\sigma)|\varphi|^{\frac{2}{n-1}} \varphi\hspace{1cm}\text{with}\hspace{1cm}\int_{M}|\varphi|^{\frac{2n}{n-1}}dv_g=1.
\end{equation}

From this  discussion, let us point out here that solutions of equation \eqref{Dirac} normalized by \\$\int_M|\varphi|^{\frac{2n}{n-1}}dv_g=1$ can be related to the B\"ar-Hijazi-Lott invariant. With this in mind, we state the main result of this note as follows

\begin{theorem}\label{mainTeo}
Let  $(M^n,g,\sigma)$ be a closed Riemannian spin manifold of positive scalar curvature. Assume $n\geq 3$ and there is an open bounded domain $\Omega\subset M$ such that $(\Omega,g)$ is conformally flat. 
If $\varphi$ is a solution of the equation \eqref{Dirac} with $\int_{M}|\varphi|^{\frac{2n}{n-1}}dv_g=1$. Then $$\lambda^2\geq \dfrac{n^2}{4}\dfrac{(n-2)}{(n-1)}\omega_n^{2/n},$$
where $\omega_n$ is the volume of the round sphere $(\mathbb{S}^n,g_0)$.



\end{theorem}
The general existence of solutions for equation \eqref{Dirac} is a very difficult task. For instance, a standard variational approach can not produce nontrivial solutions because $2n/(n+1)$ is the critical exponent in the Sobolev inequality of the spaces involved in the definition of the functional \eqref{Functional}.

However  there exists a criterion (similar to the Yamabe problem) to provide solutions to this problem. In fact if the B\"ar-Hijazi-Lott invariant satisfies 
\begin{equation}\label{stricinequality}
\lambda_{min}^+(M^n,[g],\sigma)<\frac{n}{2}\omega_n^{1/n}.
\end{equation}
Then there is a nontrivial spinor solution of \eqref{ELequation} and therefore there is a generalized metric that attain the invariant $\lambda_{min}^+(M^n,[g],\sigma)$ (see \cite{AmH} for more details).

It is well-known that $\lambda_{min}^+(M^n,[g],\sigma)\leq\frac{n}{2}\omega_n^{1/n}$  for all closed spin manifolds $(M^n,g,\sigma)$ (see \cite{AmAubin}). Furthermore the equality is attained by the round sphere of dimension $n$ i.e., $\lambda_{min}^+(\mathbb{S}^n,[g_0],\sigma_0)=(n/2)\omega_n^{1/n}$, where $\sigma_0$ is the unique spin structure on $\mathbb{S}^n$. Nevertheless it is an open problem to determine when a closed spin manifold non conformally equivalent to the round sphere satisfies or not the strict inequality \eqref{stricinequality}. 

There are very few cases where \eqref{stricinequality} it is known. For example, in the locally conformally flat setting B. Ammann, E. Humbert and B. Morel give a sufficient condition to get \eqref{stricinequality}. Specifically, if $(M^n,g,\sigma)$ is locally conformally flat with invertible Dirac operator $D_g$ and the mass-endomorphism possesses a positive eigenvalue at a point, then the inequality \eqref{stricinequality} holds (see \cite{AHM} for more details about this result). In a recent paper \cite{IsobeSphere} T. Isobe and T. Xu constructed a family of metrics $g(\epsilon)$ on sphere $\mathbb{S}^n$ such that $$\lambda_{min}^+(\mathbb{S}^n,[g(\epsilon)],\sigma_0)<\frac{n}{2}\omega_n^{1/n}$$
for $n\geq 4$ and $\epsilon>0$ small enough.
The metrics $g(\epsilon)$ restricted to some open domain $\Omega\subset\mathbb{S}^n$  are equal to the round metric $g_0$  but are not locally conformally flat. For more about the problem of strict inequality \eqref{stricinequality} see for instance \cite{Sire}.

The manifolds described above  not only admit a nontrivial solution of \eqref{ELequation}, also contain an open domain $\Omega$ for which $(\Omega,g)$ is conformally flat. Then the Theorem \eqref{mainTeo} implies that

\begin{corollary}
Let $(M^n,g,\sigma)$ be  a closed spin manifold of dimension $n\geq 3$ with positive scalar curvature. Assume that
\begin{itemize}
\item $(M^n,g)$ is locally conformally flat with a positive eigenvalue for the mass endomorphism at some point    or,
\item $(M^n,g,\sigma)=(\mathbb{S}^n,g(\epsilon),\sigma_0)$ for small $\epsilon>0$ and $g(\epsilon)$ as above.
\end{itemize} 
Then the B\"ar-Hijazi-Lott invariant satisfies that
$$\lambda_{min}^+(M^n,[g],\sigma)^2\in\left(\dfrac{(n-2)}{8(n-1)}n^2\omega_n^{2/n}, \frac{1}{4}n^2\omega_n^{2/n}\right).$$
\end{corollary}

There is another invariant for which our estimation applies. Let $C(M)$ be the set of conformal classes of metrics on closed spin manifold $M^n$. The $\tau-$invariant of $M^n$ with spin structure $\sigma$ is defined by $$\tau(M^n,\sigma):=\sup\limits_{[g]\in C(M)}\lambda^+_{min}(M^n,[g],\sigma).$$

Clearly the $\tau-$invariant for manifolds in Corollary 1.2 satisfies
$$\left(\dfrac{\tau(M^n,\sigma)}{n\omega_n^{1/n}} \right)^2\in\left(\dfrac{(n-2)}{8(n-1)}, \frac{1}{4}\right].$$ 

In the literature there are other type of lower bounds for the B\"ar-Hijazi-Lott invariant. In dimension $n\geq 3$, O. Hijazi in \cite{Hijazi}  proved that $$\lambda_{min}^+(M^n,[g],\sigma)^2\geq \frac{n}{4(n-1)}Y(M^n,[g]),$$
where $Y(M^n,[g])$ is the Yamabe constant of the conformal class $[g]$ on $M$.

This inequality exhibits a strong connection with the Yamabe problem. For instance the Hijazi's inequality provides a spinorial proof of the Yamabe problem when the inequality \eqref{stricinequality} holds. 

Therefore, estimations of the invariant $\lambda_{min}^+(M^n,[g],\sigma)$ can help to estimate the Yamabe constant from a spinorial point of view.

The plan of this note is the following. In section 2 we will recall some background material on spin manifolds and the equation \eqref{Dirac}. In section 3 we shall show all ingredients needed later, and in section 4 we deal with the proof of the main theorem.
 
\section{Preliminaries}
In this section we recall some basic facts about spin manifolds and Dirac operator. Also some results for solutions to Dirac-type equations are recalled. For more details see (\cite{AmH},\cite{HijaziBook}).

On a closed oriented Riemannian manifold $(M^n,g)$ we can define a $SO(n)$-principal bundle $P_{SO}M$ over $M$ of oriented $g-$orthonormal bases at $x\in M$. Let $f_{\alpha,\beta}$ be the transition functions. For $n\geq 3$ there exists the universal covering $\sigma:spin(n)\rightarrow SO(n)$ where $spin(n)$ is the group generated by even unit-length vector of $\mathbb{R}^n$ in the real Clifford algebra $Cl_n$ (the associative $\mathbb{R}-$algebra generated by relation $VW+WV=-2( V,W)$ for Euclidean metric $(,)$). The manifold $M$ is called spin if there is a $spin(n)-$principal bundle $P_{spin}M$ over $M$ such that it is a double covering of $P_{SO}M$ whose restriction to each fiber is $\sigma:spin(n)\rightarrow SO(n)$. This double covering from $P_{spin}M$ to $P_{SO}M$, $\sigma$, it is known as spin structure. 

There are  four special structures associated to a spin manifold $(M^n,g,\sigma)$:
\begin{enumerate}
\item A complex vector bundle $SM:=P_{spin}(M)\times_{\rho}S_n$ where $\rho:spin(n)\rightarrow Aut(S_n)$ is the restriction to $spin(n)$ of an irreducible representation $\rho:\mathbb{C}l_n\rightarrow End(S_n) $ of the complex Clifford algebra $\mathbb{C}l_n\simeq Cl_n\otimes_{\mathbb{R}} \mathbb{C}$,  $S_n\simeq\mathbb{C}^N$ and $N=2^{[n/2]}$.
\item The Clifford multiplication $m$ on $SM$ defined by
\begin{align*}
m:TM\times SM&\rightarrow SM\\
X\otimes\varphi&\mapsto X\cdot_g\varphi:=\rho(X)\varphi.
\end{align*} 
\item A Hermitian product $\langle\cdot,\cdot\rangle$ on sections of $SM$.

\item A Levi-Civita connection $\nabla$ on $SM$.
\end{enumerate}

All these structures are compatible in the following sense:
\begin{align*}
\langle X\cdot\varphi,\psi\rangle&=-\langle\varphi,X\cdot\psi\rangle, \\
X(\langle\varphi,\psi\rangle)&=\langle\nabla_X\varphi,\psi\rangle+\langle\varphi,\nabla_X\psi\rangle,\\
\nabla_X(Y\cdot\varphi)&=\nabla_XY\varphi+Y\cdot\nabla_X\varphi,
\end{align*}
for all $X,Y\in\Gamma(TM)$ and $\varphi,\psi\in\Gamma(SM)$.

On the other hand, on $\Gamma(SM)$ we have a connection Laplacian $\nabla^*\nabla$ where\\
$\nabla^*:\Gamma(Hom(TM,SM))\rightarrow\Gamma(SM)$ is the adjoint of the Levi-Civita connection $\nabla:\Gamma(SM)\rightarrow\Gamma(Hom(TM,SM))$ defined implicitly by $$\int_M\langle\nabla^*F,\varphi\rangle dv_g=\int_M\langle F,\nabla\varphi\rangle dv_g.$$
This Laplacian satisfies a very useful Bochner-type formula. For spinor field $\varphi$ compactly supported  it is known that \begin{equation}\label{Bochner}
\Delta \left(\frac{|\varphi|^2}{2}\right)=\langle\nabla\varphi,\nabla\varphi\rangle-\langle\nabla^*\nabla\varphi,\varphi\rangle,
\end{equation}
where $\Delta$ is the non-positive Laplacian-Beltrami on $(M^n,g)$.

Now identifying $\Gamma(Hom(TM,SM))$ as $\Gamma(TM\otimes SM)$ we can define the Dirac operator $D_g$ as the composition of $\nabla$ with the Clifford multiplication $m$ i.e. $D_g:=m\circ \nabla$. For a local orthonormal frame $\{E_i\}$ we have $$D_g\varphi=\sum\limits_{i=1}^nE_i\cdot_g\nabla_{E_i}\varphi.$$
For our purpose we just recall the following properties of Dirac operator 
\begin{enumerate}
\item (Green's formula) For an open domain with smooth boundary $\Omega\subset M$,
$$\int_{\Omega}\langle D_g\varphi,\psi\rangle dv_g=\int_{\Omega}\langle\varphi,D_g\psi\rangle dv_g+\int_{\partial\Omega}\langle\varphi,Z\cdot\psi\rangle dv_g^{n-1},$$
where $Z$ is a unit normal vector field on $\partial \Omega$ and $dv_g^{n-1}$ is the volume element  of $\partial\Omega$. 
\item (Schr\"odinger-Lichnerowicz formula) $$D_g^2=\nabla^*\nabla+\frac{1}{4}sc_gId_{\Gamma(SM)}.$$
\item (Penrose operator) The composition of the Levi-Civita connection $\nabla$ with the orthogonal projection on the kernel of the Clifford multiplication $m$ defines an operator $P$ known as Penrose operator. This operator satisfies that \begin{equation}\label{c}
|\nabla\varphi|^2=\frac{1}{n}|D_g\varphi|^2+|P\varphi|^2.
\end{equation}
\end{enumerate} 
We finish this section with a couple of facts about the solutions of equation \ref{Dirac}. These can be consult in Ammann's Habilitation \cite{AmH}.

\begin{remark}

\begin{itemize}
\item Due to the possible existence of nodal sets for spinor fields $\varphi$ solutions of  \ref{Dirac} by regularity theory we have that $\varphi\in C^{1,\alpha}(SM)\cap C^{\infty}(S(M  -\{\varphi^{-1}(0)\})).$
\item The equation \ref{Dirac} satisfies a weak unique continuation property. This means that the nodal set of a $C^1$ solution does not contain any nonempty open set \cite[Corollary 4.5.2]{AmH}.
\end{itemize}
\end{remark}

\section{Estimation of Parameters}

In this section we will prove some key results.

We consider the Dirac-type equation
\begin{equation}\label{Diracq}
D_g\varphi=\lambda|\varphi|^{q}\varphi
\end{equation}

where $\lambda, q\in \mathbb{R}_+$. 

\begin{lemma}
Let $\varphi$ be a solution of \eqref{Diracq} without zeros in an open domain $\Omega\subset M$. Assume $\varphi\in C^{\infty}(S\Omega)$. For any $u\in C_0^2(\Omega)$  it satisfies the following integral identity

\begin{equation}\label{eq}
0=\int_{\Omega} |\varphi|^2\left\lbrace\frac{-\Delta u}{2}+\frac{u}{4}sc_g+\left(\frac{u}{n}-u\right)\lambda^2|\varphi|^{2q}\right\rbrace dv_g + \int_{\Omega} u|P\varphi|^2dv_g.
\end{equation}

\end{lemma}

\begin{proof}
We assume a nontrivial function  $u\in C^2_0(\Omega)$. Applying the Bochner-type formula \eqref{Bochner} to $\varphi$, multiplying by $u$ and integrating we get
$$0=-\int_{\Omega} \frac{u}{2}\Delta(|\varphi|^2)dv_g+ \int_{\Omega} u|\nabla\varphi|^2dv_g-\int_{\Omega} u\langle\nabla^*\nabla\varphi,\varphi\rangle dv_g. $$

By the Schr\"odinger-Lichnerowicz formula it follows that
\begin{equation}\label{1}
0=-\int_{\Omega} \frac{u}{2}\Delta(|\varphi|^2)dv_g+ \int_{\Omega} u|\nabla\varphi|^2dv_g-\int_{\Omega} u\langle D_g^2\varphi,\varphi\rangle dv_g+\frac{1}{4}\int_{\Omega} u\; sc_g|\varphi|^2dv_g. 
\end{equation}

Using Green's formula for Dirac operator we have,
\begin{align*}
\int_{\Omega} u\langle D_g^2\varphi,\varphi\rangle dv_g&=\int_{\Omega}\langle D_g\varphi,D_g(u\varphi)\rangle dv_g+ \int_{\partial \Omega}u\langle  D_g\varphi,N\cdot\varphi\rangle dv^{n-1}_g\\
&=\int_{\Omega} \left(\langle D_g\varphi,uD_g\varphi\rangle +\langle D_g\varphi,\nabla u\cdot\varphi\rangle\right) dv_g+\int_{\partial \Omega}u\langle  D_g\varphi,N\cdot\varphi\rangle dv^{n-1}_g\\
&=\int_{\Omega} u\lambda^2|\varphi|^{2q}|\varphi|^2dv_g+\int_{\Omega}\lambda|\varphi|^q\langle\varphi,\nabla u\cdot\varphi\rangle dv_g+\int_{\partial \Omega}u\langle  D_g\varphi,N\cdot\varphi\rangle dv^{n-1}_g.
\end{align*}
In last equation we used that $\varphi$ is a solution of equation \eqref{Diracq}.

On the other hand,
\begin{align*}
\int_{\Omega}-\frac{u}{2}\Delta(|\varphi|^2)dv_g=\int_{\Omega}-\frac{\Delta u}{2}|\varphi|^2dv_g+
\int_{\partial \Omega}\left(\frac{|\varphi|^2}{2}\langle\nabla u,N\rangle-\frac{u}{2}\langle N,\nabla|\varphi|^2\rangle\right)dv_g.
\end{align*}

Since that $supp\;(u)\subset \Omega$, $u=0$ and $\nabla u=0$ on $\partial \Omega$.
Therefore all previous boundary terms vanish i.e.
\begin{align}
\int_{\Omega} u\langle D_g^2\varphi,\varphi\rangle dv_g&= \int_{\Omega} u\lambda^2|\varphi|^{2q}|\varphi|^2dv_g+\int_{\Omega}\lambda|\varphi|^q\langle\varphi,\nabla u\cdot\varphi\rangle dv_g \label{a} \\
\int_{\Omega}-\frac{u}{2}\Delta(|\varphi|^2)dv_g&=\int_{\Omega}-\frac{\Delta u}{2}|\varphi|^2dv_g\label{b}
\end{align}

Using the spliting \eqref{c} and  combining the relations \eqref{a};\eqref{b} into the equation \eqref{1} we obtain,
\begin{align*}
0&=-\int_{\Omega} \frac{u}{2}\Delta(|\varphi|^2)dv_g+ \int_{\Omega} u|\nabla\varphi|^2dv_g-\int_{\Omega} u\langle D_g^2\varphi,\varphi\rangle dv_g+\frac{1}{4}\int_{\Omega} u\; sc_g|\varphi|^2dv_g\\ 
&=-\int_{\Omega} \frac{\Delta u}{2}|\varphi|^2dv_g+ \int_{\Omega} u|P\varphi|^2dv_g+\int_{\Omega} \frac{u}{n}|D_g\varphi|^2dv_g+\frac{1}{4}\int_{\Omega} u\;sc_g|\varphi|^2dv_g\\
&-\int_{\Omega} u\lambda^2|\varphi|^{2q}|\varphi|^2dv_g-\int_{\Omega}\lambda|\varphi|^q\langle\varphi,\nabla u\cdot\varphi\rangle dv_g\\
&=\int_{\Omega} |\varphi|^2\left\lbrace\frac{-\Delta u}{2}+\frac{u}{4}sc_g+\left(\frac{u}{n}-u\right)\lambda^2|\varphi|^{2q}\right\rbrace dv_g\\
&+\int_{\Omega} u|P\varphi|^2dv_g -\int_{\Omega}\lambda|\varphi|^q\langle\varphi,\nabla u\cdot\varphi\rangle dv_g.
\end{align*}

Now note  that $\langle\varphi,\nabla u\cdot\varphi\rangle$ is a pure imaginary: $$\langle\varphi,\nabla u\cdot\varphi\rangle=-\langle\nabla u\cdot\varphi,\varphi\rangle=-\overline{\langle\varphi,\nabla u\cdot\varphi\rangle}.$$
Hence the identity \eqref{eq} follows.

\end{proof}

In the next, we will find conditions for which  the right-hand side in (\ref{eq}) is non-negative. For that purpose, we define $L(u):=-2\Delta u+sc_gu$ and $f(x):=4\left(1-\frac{1}{n}\right)|\varphi|^{2q}(x)$, where $\varphi$ is  (as in the previous Lemma 3.1) a solution of \eqref{Diracq} such that $\varphi\in C^{\infty}(S\Omega)$ and $\Omega\cap \varphi^{-1}(0)=\emptyset$.

We consider a smooth open bounded domain $\Omega'\subset\Omega\subset M$ and the weighted eigenvalue problem   \begin{equation}\label{linear}
Lu=\lambda^2f(x)u,\quad x\in \Omega',\end{equation}
with Dirichlet boundary condition $u=0$ on $\partial \Omega'$.
\begin{lemma}
There exists a $\lambda_1(f)>0$ such that
\begin{enumerate}
\item For $\lambda^2=\lambda_1(f)$ there is a positive solution of problem \eqref{linear}; 
\item for $\lambda^2<\lambda_1(f)$ there is a positive function $u_{\lambda}$ in $\Omega'$ such that $L(u_{\lambda})-\lambda^2f(x)u_{\lambda}\in C^{\infty}_+(\Omega')$ and $u_{\lambda}=0$ on $\partial\Omega'.$
\end{enumerate}
\end{lemma}
\begin{proof}
We use the standard Fredholm theory.

Since $sc_g>0$ the operator $L$ is uniformly elliptic  and self-adjoint in $H^1_0(\Omega)$ (the closure of the space $C^{\infty}_0(\Omega')$ in $L^2_1(\Omega')$). Let $K:L^2(\Omega')\rightarrow H^1_0(\Omega')$ the Green operator associated with $L$, i.e. $K(h)=u,$ where $u$ is the unique weak solution of 
\begin{align*}
L(u)&=h\quad\text{in}\;  \Omega',\\  u&=0 \quad \text{on}\;\partial\Omega'.
\end{align*}

We note that, by Sobolev embedding and $L^p$ theory, this operator is compact as a map from $H^1_0(\Omega')$ into itself.

For all $u\in H^1_0(\Omega')$ we have that $fu\in H^1_0(\Omega')$ ($f\in C^{\infty}(\Omega')$). Therefore, we can define the compact linear operator $T:H^1_0(\Omega')\rightarrow H^1_0(\Omega')$ by $T(u)=K(fu)$.

The Riesz-Fredholm theory allows us to obtain a sequence $(\mu_i(f,\Omega'))$ of non-increasing values such that
\begin{itemize}
\item $\mu_i\rightarrow0$ as $i\rightarrow\infty$.
\item $Ker(\mu_iId-T)\neq\emptyset\quad\forall i=1,2,...$
\item $\mu Id-T$ is bijective for $\mu\neq\mu_i$.
\end{itemize}

Since $f\in C^{\infty}(\Omega')$ the eigenfunctions $u_i$ associated to eigenvalues $\lambda_i(f,\Omega')=1/\mu_ i(f,\Omega')$ are classical solutions of \eqref{linear}.

Moreover, the first eigenvalue $\lambda_1(f):=\lambda_1(f,\Omega')=\frac{1}{\mu_1(f,\Omega')}$ of the weighted problem (\ref{linear}) satisfies that is simple and its eigenspace is generated by a positive eigenfunction $u_1\in H_0^1(\Omega)$. Therefore (1) follows.

On the other hand we get (2) applying the weak maximum principle: For a positive smooth function $h$ on $M$ and $\lambda^2<\lambda_1(f)$
\begin{eqnarray*}
L(u)&=\lambda^2fu+h,\quad\text{in}\quad \Omega',\\
u&=0\hspace{1.9cm}\text{on}\quad\Omega'
\end{eqnarray*}
has a positive solution $u_{\lambda}$.
\end{proof}

From the identity in Lemma 3.1 and Lemma 3.2 we obtain some restrictions on solutions of the equation \eqref{Diracq}. Indeed

\begin{theorem}\label{keybound}
Let $\varphi$ be a spinor field solution of \eqref{Diracq}. Assume that $\varphi\in C^2(S\Omega)$ for some open domain $\Omega\subset M$ such that $\Omega\cap \varphi^{-1}(0)=\emptyset$. Let $\Omega'\subset\Omega$ be a smooth open domain. Then
$$\lambda^2\geq \lambda_1\left(\frac{4(n-1)}{n}|\varphi|^{2q},\Omega'\right),$$
where $\lambda_1\left(\frac{4(n-1)}{n}|\varphi|^{2q},\Omega'\right)$ is the first positive eigenvalue of problem \eqref{linear} with Dirichlet boundary condition.
\end{theorem}
\begin{proof}
Suppose $\lambda^2<\lambda_1\left(\frac{4(n-1)}{n}|\varphi|^{2q},\Omega'\right)$. By (2) in Lemma 3.2 we have a function $u_{\lambda}\in C^2_+(\Omega')$ such that $L(u_{\lambda})-\lambda^2fu_{\lambda}(\in C^{\infty}(\Omega'))$ is strictly positive in $\Omega'$ and $u_{\lambda}=0$ on $\partial\Omega'$.

We define $u_{\lambda}=0$ in $\Omega-\Omega'$.

From identity in Lemma 3.1,
\begin{align*}
0&=\int_{\Omega} |\varphi|^2\left\lbrace\frac{-\Delta u_{\lambda}}{2}+\frac{u_{\lambda}}{4}sc_g+\left(\frac{u_{\lambda}}{n}-u_{\lambda}\right)\lambda^2|\varphi|^{2q}\right\rbrace dv_g + \int_{\Omega} u_{\lambda}|P\varphi|^2dv_g\\
0&=\int_{\Omega}|\varphi|^2\{L(u_{\lambda})-\lambda^2fu_{\lambda}\}dv_g+\int_{\Omega}4u_{\lambda}|P\varphi|^2dv_g\\
0&=\int_{\Omega'}|\varphi|^2\{L(u_{\lambda})-\lambda^2fu_{\lambda}\}dv_g+\int_{\Omega'}4u_{\lambda}|P\varphi|^2dv_g.
\end{align*} 
Hence $\varphi=0$ in $\Omega'$ and this is a contradiction.

\end{proof}

\begin{remark}
Similarly to the proof of the previous Theorem,  the case $\lambda^2=\lambda_1\left(\frac{4(n-1)|\varphi|^{2q}}{n},\Omega'\right)$ implies that $P\varphi=0$ in $\Omega'$.

We highlight that the equation $P\varphi=0$ is completely understood when $\Omega'=M$. More precisely
\begin{itemize}
\item (Lichnerowicz,\cite{Lich})($n\geq2 $) If the zero set of $\varphi$ is non-empty, then $(M^n,g,\sigma)$ is conformally equivalent to $(\mathbb{S}^n,g_0,\sigma_0)$.
\end{itemize}

Moreover, by classification of manifolds with Killing spinors and (\cite{Baum})
\begin{itemize}
\item ($n\geq3$) If the zero set of $\varphi$ is empty, then $(M^n,g)$ is conformal to an Einstein manifold.
\end{itemize}

\end{remark}

\section{Explicit lower bound}

In this section we will prove the Theorem \eqref{mainTeo}. To give so we shall see explicit lower bounds for $\lambda_1\left(\frac{4(n-1)}{n}|\varphi|^{4/(n-1)},\Omega'\right)$ in a particular family of manifolds.

\begin{proof}[Proof of Theorem \eqref{mainTeo}]
First, we recall a well-known result that will be useful in the proof.
\begin{lemma}\label{SY}[Proposition 2.2, \cite{Schoen}]
Let $(M^n,g)$ be a Riemannian manifold, $n\geq 3$. Suppose there is a domain $\Omega\subset M$ such that there exists a conformal immersion from $(\Omega,g)$ into $(\mathbb{S}^n,g_0)$. For all $u\in C^{\infty}(M)$ with $supp(u)\subset\Omega$,
$$\left(\int_{M}u^{2n/(n-2)}dv_g\right)^{(n-2)/n}\leq \dfrac{4}{n(n-2)\omega_n^{2/n}}\left(\int_M|\nabla u|^2dv_g+\dfrac{n-2}{4(n-1)}\int_Msc_gu^2dv_g\right),$$
where $\omega_n$ is the volume of the round sphere $(\mathbb{S}^n,g_0)$.
\end{lemma}

Let $\Omega\subset M$ be an open bounded domain such that $(\Omega,g)$ is conformally flat. 
We consider a solution $\varphi$ of the equation \ref{Dirac} with $\int_{M}|\varphi|^{\frac{2n}{n-1}}dv_g=1$.

From the variational characterization of eigenvalues $$\dfrac{1}{\lambda_1\left(\frac{4(n-1)}{n}|\varphi|^{4/(n-1)},\Omega\right)}=\frac{4(n-1)}{n}
\sup\left\lbrace\int_{\Omega}|\varphi|^{4/(n-1)}u^2dv_g/\int_{\Omega} uLudv_g=1\right\rbrace.$$
Using H\"older inequality we get that,
$$\int_{\Omega}|\varphi|^{4/(n-1)}u^2dv_g\leq\left(\int_{\Omega}|\varphi|^{2n/(n-1)}dv_g\right)^{2/n}\left(\int_{\Omega}u^{2n/(n-2)}dv_g\right)^{(n-2)/n}.$$
The normalization of $\varphi$ implies that $$\dfrac{1}{\lambda_1\left(\frac{4(n-1)}{n}|\varphi|^{4/(n-1)},\Omega\right)}\leq\frac{4(n-1)}{n}
\sup\left\lbrace\left(\int_{\Omega}u^{2n/(n-2)}dv_g\right)^{(n-2)/n}/\int_{\Omega} uLu dv_g=1\right\rbrace.$$
Hence
$$\frac{n}{4(n-1)}\inf\left\lbrace\dfrac{1}{\left(\int_{\Omega}u^{2n/(n-2)}dv_g\right)^{(n-2)/n}}\quad/ \int_{\Omega} uLudv_g=1\right\rbrace\leq \lambda_1\left(\frac{4(n-1)}{n}|\varphi|^{4/(n-1)},\Omega\right).$$

On the other hand, 
\begin{align*}
\inf\left\lbrace\dfrac{1}{\left(\int_{\Omega}u^{2n/(n-2)}dv_g\right)^{(n-2)/n}}\;/ \int_{\Omega} uLudv_g=1\right\rbrace&=\inf\limits_{supp(u)\subset\Omega}\dfrac{\int_{\Omega} uLudv_g}{\left(\int_{\Omega} u^{2n/(n-2)}dv_g\right)^{(n-2)/n}}\\
&=2\inf\limits_{supp(u)\subset\Omega}\dfrac{\int_{\Omega} |\nabla u|^2dv_g+\frac{1}{2}\int_{\Omega} sc_gu^2dv_g}{\left(\int_{\Omega} u^{2n/(n-2)}dv_g\right)^{(n-2)/n}}.
\end{align*}
Since that $\frac{n-2}{4(n-1)}<1/2$ we get that $$\lambda_1\left(\frac{4(n-1)}{n}|\varphi|^{4/(n-1)},\Omega\right)\geq \frac{2n}{4(n-1)}\inf\limits_{supp(u)\subset\Omega}\dfrac{\int_{\Omega} |\nabla u|^2dv_g+\frac{ n-2}{4(n-1)}\int_{\Omega}sc_gu^2dv_g}{\left(\int_{\Omega} u^{2n/(n-2)}dv_g\right)^{(n-2)/n}}.$$

Without loss of generality we can suppose that $\Omega$ is simply connected. Since $(\Omega,g)$ is conformally flat there exists a conformal immersion from $(\Omega,g)$ into $(\mathbb{S}^n,g_0)$. Then by Lemma 4.1 $$\lambda_1\left(\frac{4(n-1)}{n}|\varphi|^{4/(n-1)},\Omega\right)\geq \dfrac{n}{4(n-1)}\dfrac{n(n-2)\omega_n^{2/n}}{2}=\dfrac{n^2}{8}\dfrac{(n-2)}{(n-1)}\omega_n^{2/n}.$$

Moreover by weak unique continuation property  \cite[Corollary 4.5.2]{AmH}  we can assume that the nontrivial solution $\varphi$ does not have zero in $\Omega$. Therefore Theorem \eqref{keybound} implies $$\lambda^2\geq \dfrac{n^2}{8}\dfrac{(n-2)}{(n-1)}\omega_n^{2/n} .$$
\end{proof}

\end{document}